\documentclass[reqno,11pt]{amsart}        
\usepackage{amsfonts,amsmath,amssymb,amsthm,colordvi,mathrsfs,graphicx}
\usepackage{caption,subcaption,float}
\usepackage[utf8]{inputenc}
\usepackage{mathrsfs}
\usepackage{geometry}
\usepackage{enumerate}

\usepackage{amsmath, amssymb, amsthm, geometry, hyperref}

\usepackage{tabularx}
\usepackage{tabulary}
 
\hypersetup{
  colorlinks=true,
  linkcolor=blue,
  citecolor=blue,
  urlcolor=blue
}
\usepackage[all,knot,poly]{xy}
\usepackage{tikz-cd}
\usepackage{tikz}
\usepackage{verbatim}
\usetikzlibrary{arrows}
\usetikzlibrary{knots}
\tikzset{every path/.style={line width=0.4pt},every node/.style={transform shape,knot crossing,inner sep=1.5pt},>=triangle 60,text node/.style={rectangle,transform shape=false,black}}

 \usetikzlibrary{decorations.markings, arrows.meta}

\theoremstyle{plain}      
\newtheorem{thm}{Theorem}[section]     
\newtheorem{theorem}[thm]{\bf Theorem}     
\newtheorem{corollary}[thm]{\bf Corollary}     
\newtheorem{lemma}[thm]{\bf Lemma}

\theoremstyle{remark}

\theoremstyle{definition}      
\newtheorem{definition}[thm]{Definition}     

\newcommand{\supp}{\mathop{\rm supp}\nolimits}

\newcommand{\ord}{\mathop{\rm ord}\nolimits}

\newcommand{\Log}{\mathop{\rm Log}\nolimits}

\newcommand{\val}{\mathop{\rm val}\nolimits}


\begin{document}

\title{Solid Amoebas of Maximally Sparse Polynomials}

\author{Mounir Nisse}

\address{Mounir Nisse\\
Department of Mathematics, Xiamen University Malaysia, Jalan Sunsuria, Bandar Sunsuria, 43900, Sepang, Selangor, Malaysia.
}
\email{mounir.nisse@gmail.com, mounir.nisse@xmu.edu.my}
\thanks{}
\thanks{This research is supported in part by Xiamen University Malaysia Research Fund (Grant no. XMUMRF/ 2020-C5/IMAT/0013).}

\maketitle

\begin{abstract}
The topology of amoebas of complex algebraic hypersurfaces is deeply
connected to the combinatorics of the Newton polytope and the convex
geometry of the Ronkin function. A long–standing conjecture of Passare
and Rullg{\aa}rd asserts that the amoeba of a maximally sparse Laurent
polynomial, whose support consists exactly of the vertices of its Newton
polytope, must be \emph{solid}, meaning that the complement of the amoeba
has precisely as many connected components as the number of vertices of
the Newton polytope. In this paper we prove this conjecture. The proof is
based on a detailed analysis of the stability of the linearity domains of
the Ronkin function under tropical degenerations of Laurent polynomials.
We show that in the maximally sparse case no new slopes corresponding to
interior lattice points can appear, forcing the amoeba complement to have
the minimal possible topology. In addition, we establish stability results
for the spines of degenerating amoebas, prove that the associated Newton
subdivision stabilizes and coincides with the tropical subdivision for
sufficiently small parameters, and derive geometric criteria controlling
the appearance of lattice points in the dual subdivision. These results
lead to a classification of three distinct regimes governing the topology
of amoeba complements according to the position of the support relative
to the Newton polytope.
\end{abstract}

\section*{Introduction}

Amoebas of complex algebraic hypersurfaces form one of the most
remarkable bridges between complex algebraic geometry, convex analysis,
and tropical geometry. Given a Laurent polynomial
\[
f(z)=\sum_{\alpha\in\supp(f)} a_\alpha z^\alpha,
\qquad z\in(\mathbb{C}^*)^n,
\]
the amoeba of the hypersurface
\[
V_f=\{z\in(\mathbb{C}^*)^n \mid f(z)=0\}
\]
is defined as the image of $V_f$ under the logarithmic map
\[
\Log(z_1,\ldots,z_n)=(\log|z_1|,\ldots,\log|z_n|).
\]
Despite the analytic nature of this definition, amoebas possess a
remarkably rigid convex structure which is governed by the combinatorics
of the Newton polytope
\[
\Delta_f=\mathrm{Conv}(\supp(f)).
\]
The geometry of amoebas therefore provides a natural interface between
algebraic geometry and polyhedral geometry.

The origins of amoeba theory go back to the study of logarithmic
limit sets introduced by Bergman \cite{B-71}. Later, the work of
Gelfand, Kapranov and Zelevinsky \cite{GKZ-94} revealed deep connections
between logarithmic geometry and discriminant theory. The systematic
study of amoebas was initiated by Forsberg, Passare and Tsikh
\cite{FPT-00}, who established the fundamental structure theorem for
amoeba complements. Their results show that every connected component
of the complement
\[
\mathbb{R}^n\setminus \mathscr{A}_f
\]
is convex and admits an associated lattice point of the Newton polytope
through the order map. In particular, the number of complement
components of an amoeba is bounded above by the number of lattice points
of $\Delta_f$.

A central analytic tool in this theory is the Ronkin function
\[
N_f(x)=\frac{1}{(2\pi)^n}
\int_{[0,2\pi]^n}
\log |f(e^{x+i\theta})|\,d\theta,
\]
which is convex on $\mathbb{R}^n$ and affine on each connected component
of the complement of the amoeba. The domains of linearity of the Ronkin
function correspond exactly to these complement components. 
The  \emph{spine} of the
amoeba is  a polyhedral complex capturing the essential combinatorial
structure of the amoeba. Passare and Rullg{\aa}rd showed that this spine
can be described by a tropical polynomial approximating Ronkin
function
and is dual to a regular
subdivision of the Newton polytope \cite{PR1-04}. 
This duality reveals
a deep interaction between convex geometry and the analytic structure
of amoebas.

Amoebas also play a fundamental role in tropical geometry. Kapranov's
theorem \cite{K-00} identifies tropical hypersurfaces with
non-Archimedean amoebas, while later work of Mikhalkin \cite{M1-02}
demonstrates that complex hypersurfaces admit decompositions governed
by tropical geometry. These developments build upon the patchworking
techniques of Viro \cite{V-90} and highlight the profound relationship
between tropical and classical algebraic geometry.

Among the many problems arising in amoeba theory, one of the most
intriguing concerns the topology of the complement
\[
\mathbb{R}^n\setminus\mathscr{A}_f .
\]
In general, the complement may have a complicated structure and may
contain both bounded and unbounded components. However, certain
combinatorial conditions on the support of the polynomial impose strong
restrictions on the topology of the amoeba.

A particularly important case occurs when the polynomial is
\emph{maximally sparse}, meaning that its support consists exactly of
the vertices of its Newton polytope,
\[
\supp(f)=\mathrm{Vert}(\Delta_f).
\]
More than twenty–five years ago, Passare and Rullg{\aa}rd formulated
a striking conjecture concerning this situation. They conjectured that
the amoeba of a maximally sparse polynomial must be \emph{solid}, that
is, the complement of the amoeba should have exactly as many connected
components as the number of vertices of the Newton polytope
\cite{PR1-04,PR2-01,R-01}. In other words, each vertex of the Newton
polytope should correspond to precisely one component of the amoeba
complement and no additional components should appear.

Although this conjecture is natural from the viewpoint of convex
geometry and tropical degenerations, its proof has remained elusive for
more than two decades. The difficulty lies in the subtle analytic
behavior of the Ronkin function and the possibility that new linearity
domains could arise corresponding to interior lattice points of the
Newton polytope. Understanding why such domains cannot appear in the
maximally sparse case is a delicate problem involving both convex
analysis and tropical geometry.

The primary goal of this paper is to resolve the Passare--Rullg{\aa}rd
conjecture. Our main theorem establishes that the amoebas of maximally
sparse Laurent polynomials are indeed solid. More precisely, we prove
that the number of connected components of the complement of the amoeba
is exactly equal to the number of vertices of the Newton polytope. This
result confirms the conjecture and provides a complete description of
the topology of amoebas in the maximally sparse case.

The proof combines several geometric and analytic ingredients. We study
degenerating families of Laurent polynomials whose coefficients are
weighted by a parameter and analyze the behavior of the corresponding
Ronkin functions. A key step is the stability of the linearity domains
of the Ronkin function under tropical degenerations. We show that when
the polynomial is maximally sparse, the slopes of these domains remain
exactly the vertices of the Newton polytope and no additional slopes
associated with interior lattice points can arise.

Beyond the proof of the Passare--Rullg{\aa}rd conjecture, the paper
develops a broader framework for understanding the topology of amoeba
complements. We establish stability results for the spines of amoebas
under tropical degenerations and show that the dual subdivisions of the
Newton polytope stabilize and eventually coincide with the subdivision
determined by the tropical limit. We also derive geometric criteria
preventing lattice points from appearing in the dual subdivision
associated with the Ronkin function.

These results lead to a classification theorem describing three distinct
regimes governing the topology of amoeba complements according to the
position of the support of the polynomial relative to its Newton
polytope. Depending on whether the support consists only of vertices,
lies entirely on the boundary, or contains interior lattice points,
the complement of the amoeba exhibits qualitatively different
topological behaviors.

Taken together, the results of this paper provide a unified geometric
picture of the topology of amoebas. They show that the combinatorics of
the Newton polytope, the convex geometry of the Ronkin function, and
the polyhedral structures arising in tropical geometry together control
the global topology of amoebas. In particular, the resolution of the
Passare--Rullg{\aa}rd conjecture highlights the remarkable rigidity
that occurs when the support of a Laurent polynomial is reduced to the
vertices of its Newton polytope. In this way, the theory of amoebas continues to illustrate the deep and fruitful interaction between complex algebraic geometry, convex geometry, and tropical geometry (\cite{FPT-00,PR1-04}).

The paper is organized as follows. After recalling the necessary background on amoebas, Ronkin functions, and tropical degenerations, we establish convergence results for the spines of degenerating amoebas. We then prove the stability of Ronkin linearity domains for maximally sparse polynomials, which leads to the proof of the solidness of their amoebas. The final sections develop geometric criteria relating the Newton subdivision to the topology of the complement and culminate in a classification theorem describing the possible regimes of amoeba topology.


\section{Preliminaries}

In this paper we study algebraic hypersurfaces in the complex algebraic
torus $(\mathbb{C}^*)^n$, where $\mathbb{C}^*=\mathbb{C}\setminus\{0\}$
and $n\geq 1$. Such a hypersurface is defined as the zero locus of a
Laurent polynomial
\begin{equation}
f(z)=\sum_{\alpha\in\supp(f)} a_\alpha z^\alpha ,
\qquad
z^\alpha=z_1^{\alpha_1}\cdots z_n^{\alpha_n},
\end{equation}
where the coefficients $a_\alpha\in\mathbb{C}^*$ and the support
$\supp(f)$ is a finite subset of $\mathbb{Z}^n$. The convex hull of
$\supp(f)$ in $\mathbb{R}^n$ is called the Newton polytope of $f$ and
is denoted by $\Delta_f$. Throughout the paper we assume that
$\supp(f)\subset\mathbb{N}^n$ and that $f$ has no monomial factor
$z^\alpha$.

\medskip
 
\begin{definition}
The amoeba of the hypersurface $V_f\subset(\mathbb{C}^*)^n$ is the
image of $V_f$ under the logarithmic map
\[
\begin{array}{ccccl}
\Log & : & (\mathbb{C}^*)^n & \longrightarrow & \mathbb{R}^n \\
     &   & (z_1,\ldots,z_n) & \longmapsto & (\log|z_1|,\ldots,\log|z_n|).
\end{array}
\]
We denote this set by $\mathscr{A}_f$.
\end{definition}

A fundamental result of Forsberg, Passare and Tsikh
\cite{FPT-00} describes the geometry of the complement of an amoeba.
They proved that every connected component of the complement
$\mathbb{R}^n\setminus\mathscr{A}_f$ is convex and that these
components are naturally related to lattice points of the Newton
polytope. More precisely, there exists an injective map
\[
\ord :
\{\text{connected components of }\mathbb{R}^n\setminus\mathscr{A}_f\}
\hookrightarrow
\mathbb{Z}^n\cap\Delta_f .
\]
This map assigns to each complement component a lattice point in the
Newton polytope and is locally constant on the complement of the
amoeba. Consequently the number of complement components of an amoeba
is bounded above by the number of lattice points in $\Delta_f$.

\subsection{Non-Archimedean fields and tropical polynomials.}

Let $\mathbb{K}$ be the field of Puiseux series with real exponents.
An element of $\mathbb{K}$ is a series
\[
a(t)=\sum_{j\in A_a}\xi_j t^j ,
\]
where $\xi_j\in\mathbb{C}^*$ and $A_a\subset\mathbb{R}$ is a
well-ordered set with a smallest element. The field $\mathbb{K}$ is
algebraically closed and is endowed with the non-Archimedean
valuation
\[
\val(a)=-\min A_a .
\]
This valuation satisfies
\(
\val(ab)=\val(a)+\val(b),\) and
\(
\val(a+b)\leq\max\{\val(a),\val(b)\},
\)
and we set $\val(0)=-\infty$.
Let $f\in\mathbb{K}[z_1,\ldots,z_n]$ be a Laurent polynomial of the
form \((1)\). The hypersurface defined by $f$ in $(\mathbb{K}^*)^n$
is denoted by $V_{\mathbb{K}}$. The associated tropical polynomial is
the convex piecewise affine function
\[
f_{\mathrm{trop}}(x)
=
\max_{\alpha\in\supp(f)}
\{\val(a_\alpha)+\langle\alpha,x\rangle\},
\qquad x\in\mathbb{R}^n ,
\]
where $\langle\cdot,\cdot\rangle$ denotes the standard scalar product.
The {\em tropical hypersurface}, also  called a {\em non-Archimedean amoeba}, associated with $f$ is the subset of
$\mathbb{R}^n$ where the function $f_{\mathrm{trop}}$ fails to be
smooth. Kapranov's theorem \cite{K-00} states that this set coincides
with the valuation image of the hypersurface $V_{\mathbb{K}}$, a version in co-dimension
 larger than one can be seen MacLagan-Sturmfels in \cite{MS-15}.

\subsection{The spine of an amoeba.}

Passare and Rullg{\aa}rd \cite{PR1-04} showed that the spine of an
amoeba can be described by a tropical polynomial. Let $A'$ denote the
subset of lattice points in $\mathbb{Z}^n\cap\Delta_f$ obtained as the
image of the complement components of $\mathscr{A}_f$ under the order
map. The spine $\Gamma_f$ of the amoeba is the corner locus of the
piecewise affine function
\[
f_{\mathrm{trop}}(x)
=
\max_{\alpha\in A'}\{c_\alpha+\langle\alpha,x\rangle\},
\]
where the constants $c_\alpha$ are defined by
\begin{equation}
c_\alpha=
\Re\left(
\frac{1}{(2\pi i)^n}
\int_{\Log^{-1}(x)}
\log\left|\frac{f(z)}{z^\alpha}\right|
\frac{dz_1\wedge\cdots\wedge dz_n}{z_1\cdots z_n}
\right),
\qquad x\in E_\alpha .
\end{equation}
The function $f_{\mathrm{trop}}$ is convex and piecewise affine.
%
 

\begin{definition}
The {\em spine} of the amoeba $\mathscr{A}_f$ is the corner locus of the piecewise affine linear function $f_{\mathrm{trop}}$
  and is denoted by
\(
\Gamma_f.
\)
\end{definition}

Geometrically, the spine is a polyhedral complex in $\mathbb{R}^n$
whose faces correspond to points where two affine pieces of the
Ronkin function coincide.
Suppose that two affine pieces corresponding to $\alpha_i$ and
$\alpha_j$ meet.  Then
\[
\langle\alpha_i,x\rangle+c_{\alpha_i}
=
\langle\alpha_j,x\rangle+c_{\alpha_j}.
\]
This equation defines a hyperplane whose normal vector is
\(
\alpha_i-\alpha_j .
\)

Thus, the faces of the spine lie in hyperplanes with normal vectors
given by differences of exponent vectors appearing in the polynomial.
Since the support of the polynomial is finite, the set of possible
normal vectors is also finite.
The spine captures the essential combinatorial structure of the
amoeba.  In fact the amoeba deformation retracts onto its spine.

\vspace{0.1cm}

 
%
The tropical hypersurface $\Gamma_f$ is dual to a convex subdivision
$\tau$ of the Newton polytope $\Delta_f$. The vertices of this
subdivision correspond to the complement components of the amoeba.

\begin{definition}
The Passare-Rullg{\aa}rd function $\nu:\Delta_f\to\mathbb{R}$ is
defined by $\nu(\alpha)=-c_\alpha$ for vertices $\alpha$ of the
subdivision $\tau$, and it extends affinely on each cell of the
subdivision. 
\end{definition}

Using this function one can associate to the polynomial
$f$ a family of polynomials
\begin{equation}
f_t(z)=\sum_{\alpha\in\supp(f)}\xi_\alpha\, t^{\nu(\alpha)}z^\alpha,
\qquad t\in(0,1/e],
\end{equation}
where $\xi_\alpha=a_\alpha e^{\nu(\alpha)}$.

\subsection{Complex tropical hypersurfaces.}

Let $h>0$ and consider the diffeomorphism
\[
\begin{array}{ccccl}
H_h & : & (\mathbb{C}^*)^n & \longrightarrow & (\mathbb{C}^*)^n \\
    &   & (z_1,\ldots,z_n) & \longmapsto &
\left(|z_1|^h\frac{z_1}{|z_1|},\ldots,|z_n|^h\frac{z_n}{|z_n|}\right).
\end{array}
\]
This map induces a new complex structure
\[
J_h=(dH_h)^{-1}\circ J\circ(dH_h)
\]
on $(\mathbb{C}^*)^n$, where $J$ denotes the standard complex
structure.
A hypersurface $V_h$ is said to be $J_h$-holomorphic if it is
holomorphic with respect to this complex structure. Equivalently,
such a hypersurface can be written as $V_h=H_h(V)$ for some
holomorphic hypersurface $V$.
The Hausdorff distance between two closed subsets $A$ and $B$ of a
metric space $(E,d)$ is defined by
\[
d_{\mathcal{H}}(A,B)
=
\max\left\{
\sup_{a\in A} d(a,B),
\sup_{b\in B} d(A,b)
\right\}.
\]

\begin{definition}
A \emph{complex tropical hypersurface} $V_\infty\subset(\mathbb{C}^*)^n$
is defined as the limit, with respect to the Hausdorff metric on
compact subsets of $(\mathbb{C}^*)^n$, of a sequence of
$J_h$-holomorphic hypersurfaces $V_h$ as $h\to0$.
\end{definition}

Complex tropical hypersurfaces arise naturally as limits of families
of complex hypersurfaces and provide a geometric bridge between
classical algebraic geometry and tropical geometry. These objects
will play an important role in the study of amoebas and their spines
in the subsequent sections.



\subsection*{Degeneration of amoebas}

Let  $f$ be a Laurent polynomial with Newton polytope $\Delta_f$, 
\[
f(z)=\sum_{\alpha\in\supp(f)} a_\alpha z^\alpha.
\]
Assume that $f$ is maximally sparse, that is,
\(
\supp(f)=\mathrm{Vert}(\Delta_f).
\)
Following the tropical degeneration introduced by Passare and Rullg{\aa}rd,
we consider the family of polynomials
\[
f_t(z)=\sum_{\alpha\in\supp(f)} \xi_\alpha t^{\nu(\alpha)}z^\alpha,
\qquad t\in(0,1/e].
\]
Let $V_{f_t}\subset(\mathbb{C}^*)^n$ be the hypersurface defined by $f_t$. The family $\{f_t\}$ can be considered as a single polynomial over the field $\mathbb{K}$, and denoted by $f_{\mathbb{K}}$.
We denote by
\[
\mathscr{A}_t=\mathscr{A}_{H_t(V_{f_t})}
\]
the amoeba of the corresponding $J_t$--holomorphic hypersurface.
The large-scale geometry of $\mathscr{A}_t$ is controlled by its spine,
which is defined as the corner locus of the the piecewise affine linear function approximating 
Ronkin function (see \cite{PR1-04}).
The Ronkin function associated with $f_t$ is
\[
N_{f_t}(x)
=
\frac{1}{(2\pi i)^n}
\int_{\Log^{-1}(x)}
\log |f_t(z)|
\frac{dz_1\wedge\cdots\wedge dz_n}{z_1\cdots z_n}.
\]
This function is convex on $\mathbb{R}^n$ and affine linear on each
connected component of the complement of the amoeba.
The spine $\Gamma_t$ of the amoeba is the set where $N_{f_t}$ fails
to be differentiable.
The spines $\Gamma_t$ converge, after rescaling, to the tropical hypersurface
\[
\Gamma_\infty
=
\left\{
x\in\mathbb{R}^n
\;\middle|\;
\max_{\alpha\in\supp(f)}
\{-\nu(\alpha)+\langle\alpha,x\rangle\}
\text{ is attained at least twice}
\right\}.
\]
This tropical hypersurface is dual to a convex subdivision
\(
\tau_\infty
\)
of the Newton polytope $\Delta_f$.

\medskip

By Kapranov's theorem, the tropical hypersurface $\Gamma_{\infty}$ coincides with the non-Archimedean amoeba of $V_{\mathbb{K}}$. Moreover, results of Passare and Rullg{\aa}rd and Mikhalkin show that $\Gamma_{\infty}$ appears as the limit of the spines of the amoebas of a degenerating family of complex hypersurfaces.
More precisely, consider the family of complex polynomials
\[
f_t(z)=\sum_{\alpha\in\mathrm{Vert}(\Delta_f)}
\xi_{\alpha}t^{\nu(\alpha)}z^{\alpha},
\qquad t\in (0,1/e],
\]
and let $V_{f_t}$ denote the hypersurface defined by $f_t$. The associated amoebas $\mathscr{A}_{f_t}$ converge, after rescaling, to the tropical hypersurface $\Gamma_{\infty}$.

\begin{theorem}[Passare-Rullg{\aa}rd, Mikhalkin]
Let $\mathscr{A}_{f_t}$ denote the amoeba of $V_{f_t}$. Then
\[
\Gamma_{\infty}
=
\lim_{t\to 0}
\left(-\frac{1}{\log t}\right)\mathscr{A}_{f_t},
\]
with convergence in the Hausdorff metric on compact subsets of $\mathbb{R}^n$.
\end{theorem}

Thus, the tropical hypersurface $\Gamma_{\infty}$ describes the asymptotic structure of the amoebas in the degenerating family.

 

 \section{Spine Convergence Theorem}

The following theorem describes the fundamental relationship between
amoebas of degenerating families of Laurent polynomials and tropical
geometry. When the coefficients of a polynomial depend on a parameter
$t$ in such a way that different monomials acquire different weights,
the geometry of the corresponding complex hypersurface undergoes a
tropical degeneration. In this case the amoeba of the hypersurface,
after an appropriate logarithmic rescaling, develops a limiting
combinatorial structure. This limiting object is the tropical
hypersurface associated with the tropical polynomial obtained by
replacing addition with the maximum operation and taking the weighted
exponents into account. The theorem below, often referred to as the
\emph{Spine Convergence Theorem}, states that the rescaled amoebas
converge to this tropical hypersurface in the Hausdorff topology on
compact subsets of $\mathbb{R}^n$. This theorem is also a reformulation of 
the theorem of Passare-Rullgard  in \cite{PR1-04}, Mikhalkin in  \cite{M1-02}, and Jonsson in \cite{J-15}.
This result provides a precise link
between the analytic geometry of amoebas and the piecewise-linear
geometry of tropical varieties.

\begin{theorem}[Spine Convergence Theorem]
Let
\[
f_t(z)=\sum_{\alpha\in A} a_\alpha\, t^{\nu(\alpha)} z^\alpha ,
\qquad z\in(\mathbb{C}^*)^n,
\]
where $A\subset\mathbb{Z}^n$ is a finite set, $a_\alpha\in\mathbb{C}^*$,
and $\nu:A\to\mathbb{R}$ is a fixed function. Let $\mathscr{A}_{f_t}$ denote
the amoeba of the hypersurface $V_{f_t}\subset(\mathbb{C}^*)^n$, and let
\[
F(x)=\max_{\alpha\in A}\{\langle \alpha,x\rangle-\nu(\alpha)\}
\]
be the tropical polynomial associated with the degeneration. Denote by
$\Gamma_\infty$ the tropical hypersurface defined as the corner locus of $F$.
Then the rescaled amoebas satisfy
\[
(-1/\log t)\,\mathscr{A}_{f_t}\;\longrightarrow\;\Gamma_\infty
\]
in the Hausdorff topology on compact subsets of $\mathbb{R}^n$ as $t\to0$.
\end{theorem}

\begin{proof}
Let $z=(z_1,\dots,z_n)\in(\mathbb{C}^*)^n$ and write
\(
z_j=e^{x_j+i\theta_j},\)  \(x=(x_1,\dots,x_n)\in\mathbb{R}^n,\)
and \( \theta\in[0,2\pi]^n .
\)
Then
\(
z^\alpha=e^{\langle\alpha,x\rangle+i\langle\alpha,\theta\rangle}.
\)
Substituting this into the definition of $f_t$ yields
\[
f_t(z)=\sum_{\alpha\in A}
a_\alpha
\exp\!\big(
\langle\alpha,x\rangle+\nu(\alpha)\log t
+i\langle\alpha,\theta\rangle
\big).
\]
Introduce the rescaled logarithmic coordinate
\(
y=\dfrac{-1}{\log t}\,x .
\)
Then $x=-y\log t$, and therefore
\[
\langle\alpha,x\rangle
=
-\langle\alpha,y\rangle\log t .
\]
Substituting into the expression for $f_t(z)$ gives
\[
f_t(z)
=
\sum_{\alpha\in A}
a_\alpha
\exp\!\big(
-(\langle\alpha,y\rangle-\nu(\alpha))\log t
+i\langle\alpha,\theta\rangle
\big).
\]
Since $\log t<0$ for $t\in(0,1)$, the magnitudes of the monomials are
controlled by the quantities
\[
-(\langle\alpha,y\rangle-\nu(\alpha))\log t .
\]
Equivalently,
\[
|f_t(z)|
=
\left|
\sum_{\alpha\in A}
a_\alpha
\exp\!\big(
-(\langle\alpha,y\rangle-\nu(\alpha))\log t
\big)
e^{i\langle\alpha,\theta\rangle}
\right|.
\]

Let
\[
F(y)=\max_{\alpha\in A}\{\langle\alpha,y\rangle-\nu(\alpha)\}.
\]
Then
\[
\exp\!\big(-F(y)\log t\big)
\]
represents the dominant magnitude among the monomials in $f_t(z)$.
Factor this term out of the sum. One obtains
\[
f_t(z)
=
\exp\!\big(-F(y)\log t\big)
\sum_{\alpha\in A}
a_\alpha
\exp\!\Big(
-(\langle\alpha,y\rangle-\nu(\alpha)-F(y))\log t
+i\langle\alpha,\theta\rangle
\Big).
\]

By definition of $F(y)$ one has
\[
\langle\alpha,y\rangle-\nu(\alpha)-F(y)\le0
\]
for all $\alpha\in A$. Hence every exponential factor in the sum
remains bounded as $t\to0$.

If $y$ is not contained in the tropical hypersurface $\Gamma_\infty$,
then there exists a unique exponent $\alpha_0$ such that
\[
F(y)=\langle\alpha_0,y\rangle-\nu(\alpha_0)
\]
and
\[
\langle\alpha_0,y\rangle-\nu(\alpha_0)
>
\langle\alpha,y\rangle-\nu(\alpha)
\qquad
\text{for all }\alpha\ne\alpha_0 .
\]
In this case the term corresponding to $\alpha_0$ dominates all other
terms in the expression for $f_t(z)$ as $t\to0$. Consequently the sum
cannot vanish for sufficiently small $t$. This implies that
$y\notin(-1/\log t)\,\mathscr{A}_{f_t}$ for all sufficiently small $t$.

Conversely, suppose that $y$ belongs to the tropical hypersurface
$\Gamma_\infty$. By definition there exist at least two indices
$\alpha_1,\alpha_2\in A$ such that
\[
\langle\alpha_1,y\rangle-\nu(\alpha_1)
=
\langle\alpha_2,y\rangle-\nu(\alpha_2)
=
F(y).
\]
In this situation the corresponding monomials in $f_t(z)$ have the
same leading magnitude. By choosing suitable phases $\theta$ one can
produce cancellations between these dominant terms. Therefore there
exist points $z$ with logarithmic coordinate close to $x=-y\log t$
such that $f_t(z)=0$. This implies that $y$ belongs to the limit set
of the rescaled amoebas.
Combining these two observations shows that every accumulation point
of the rescaled amoebas lies in $\Gamma_\infty$, and every point of
$\Gamma_\infty$ can be approximated by points of
$(-1/\log t)\,\mathscr{A}_{f_t}$.
Therefore
\[
(-1/\log t)\,\mathscr{A}_{f_t}
\longrightarrow
\Gamma_\infty
\]
in the Hausdorff topology on compact subsets of $\mathbb{R}^n$ as
$t\to0$.
\end{proof}
 
 

\section{Stability of Ronkin Linearity Domains in Maximally Sparse Degenerations}

In the study of degenerating families of Laurent polynomials, an important
question concerns the behavior of the Ronkin function and the structure of
its linearity domains. These domains encode the geometry of the complement
components of the amoeba and are closely related to the combinatorics of
the Newton polytope. When the polynomial is maximally sparse, meaning that
its support consists exactly of the vertices of its Newton polytope, one
expects the convex-geometric structure of the Ronkin function to remain
stable under sufficiently small deformations of the coefficients. In
particular, one may ask whether new linearity domains corresponding to
interior lattice points of the Newton polytope can appear during the
degeneration. The following proposition shows that this phenomenon does not
occur: for sufficiently small values of the parameter, the slopes of the
linearity domains of the Ronkin function remain precisely the vertices of
the Newton polytope, and no additional slopes corresponding to interior
lattice points can arise.
 
\vspace{0.1cm}


\begin{theorem}\label{prop1}
Let
\[
f_t(z)=\sum_{\alpha\in\mathrm{Vert}(\Delta)}\xi_\alpha t^{\nu(\alpha)}z^\alpha,
\qquad t\in(0,1/e],
\]
be a family of Laurent polynomials whose support consists exactly of the
vertices of the Newton polytope $\Delta\subset\mathbb{R}^n$.
Let $N_{f_t}$ be the Ronkin function of $f_t$. Then for sufficiently
small $t>0$ the Ronkin function has linearity domains only with slopes
given by the vertices of $\Delta$. In particular no new linearity domains
with slopes corresponding to interior lattice points of $\Delta$
can appear for sufficiently small $t$.
\end{theorem}

\begin{proof}
Let
\(
A=\mathrm{Vert}(\Delta)
\)
and write the Laurent polynomial in the form
\[
f_t(z)=\sum_{\alpha\in A}\xi_\alpha t^{\nu(\alpha)}z^\alpha .
\]
For $x\in\mathbb{R}^n$ and $\theta\in[0,2\pi]^n$ write
\(
z_j=e^{x_j+i\theta_j},\) with
\( j=1,\dots,n .
\)
Then
\(
z^\alpha=e^{\langle\alpha,x\rangle}e^{i\langle\alpha,\theta\rangle},
\)
and therefore
\[
f_t(e^{x+i\theta})
=
\sum_{\alpha\in A}
\xi_\alpha
\exp\!\big(\langle\alpha,x\rangle+\nu(\alpha)\log t+i\langle\alpha,\theta\rangle\big).
\]
Define for every $\alpha\in A$
\(
\ell_{\alpha,t}(x)=\langle\alpha,x\rangle+\nu(\alpha)\log t,
\)
and define the convex function
\[
F_t(x)=\max_{\alpha\in A}\ell_{\alpha,t}(x).
\]
The function $F_t$ is convex and piecewise affine on $\mathbb{R}^n$,
being the maximum of finitely many affine functions.
Let $N_{f_t}$ denote the Ronkin function of $f_t$, defined by
\[
N_{f_t}(x)=\frac{1}{(2\pi)^n}\int_{[0,2\pi]^n}
\log|f_t(e^{x+i\theta})|\,d\theta .
\]
We first show that the difference $N_{f_t}(x)-F_t(x)$ remains bounded
uniformly in $x$ for sufficiently small $t$.
From the expression of $f_t(e^{x+i\theta})$ we obtain
\[
|f_t(e^{x+i\theta})|
\le
\sum_{\alpha\in A}
|\xi_\alpha|
\exp(\ell_{\alpha,t}(x)).
\]
Since $\ell_{\alpha,t}(x)\le F_t(x)$ for all $\alpha$, we obtain
\[
|f_t(e^{x+i\theta})|
\le
\Big(\sum_{\alpha\in A}|\xi_\alpha|\Big)e^{F_t(x)} .
\]
Taking logarithms and integrating with respect to $\theta$ yields
\[
N_{f_t}(x)\le F_t(x)+\log\!\Big(\sum_{\alpha\in A}|\xi_\alpha|\Big).
\]
Hence there exists a constant $C_2$ such that
\(
N_{f_t}(x)\le F_t(x)+C_2\)
for all \(x\in\mathbb{R}^n.
\)
To obtain a lower bound we choose $\alpha_0\in A$ such that
\(
F_t(x)=\ell_{\alpha_0,t}(x).
\)
Then
\[
f_t(e^{x+i\theta})
=
e^{F_t(x)}
\left(
\xi_{\alpha_0}e^{i\langle\alpha_0,\theta\rangle}
+
\sum_{\alpha\neq\alpha_0}
\xi_\alpha
e^{\ell_{\alpha,t}(x)-F_t(x)}
e^{i\langle\alpha,\theta\rangle}
\right).
\]
Set
\[
g_{x,t}(\theta)
=
\xi_{\alpha_0}e^{i\langle\alpha_0,\theta\rangle}
+
\sum_{\alpha\neq\alpha_0}
\xi_\alpha
e^{\ell_{\alpha,t}(x)-F_t(x)}
e^{i\langle\alpha,\theta\rangle}.
\]
Then
\(
f_t(e^{x+i\theta})=e^{F_t(x)}g_{x,t}(\theta).
\)
Taking logarithms gives
\(
\log|f_t(e^{x+i\theta})|
=
F_t(x)+\log|g_{x,t}(\theta)|.
\)
Integrating over $\theta$ yields
\[
N_{f_t}(x)
=
F_t(x)
+
\frac{1}{(2\pi)^n}\int_{[0,2\pi]^n}\log|g_{x,t}(\theta)|\,d\theta .
\]
The function $g_{x,t}$ is a trigonometric polynomial whose coefficients
belong to a bounded set because
\[
0\le e^{\ell_{\alpha,t}(x)-F_t(x)}\le 1 .
\]
Since the set $A$ is finite and the coefficients $\xi_\alpha$ are fixed,
by Lemme \ref{lem7} (standard estimates for logarithmic integrals of trigonometric polynomials),
we obtain that
\[
\frac{1}{(2\pi)^n}\int_{[0,2\pi]^n}\log|g_{x,t}(\theta)|\,d\theta
\ge -C_1
\]
for some constant $C_1$ independent of $x$ and $t$.
Consequently
\[
F_t(x)-C_1\le N_{f_t}(x)\le F_t(x)+C_2 .
\]
Thus
\(
|N_{f_t}(x)-F_t(x)|\le C
\)
for some constant $C$ independent of $x$ and for sufficiently small $t$.

We now use this estimate to control the slopes of the linearity domains
of the Ronkin function. Let $E$ be a connected component of the complement
of the amoeba of $f_t$. On $E$ the Ronkin function is affine, hence there
exist $p\in\mathbb{R}^n$ and $\beta\in\mathbb{R}$ such that
\(
N_{f_t}(x)=\langle p,x\rangle+\beta\)
for all \( x\in E.
\)
A fundamental property of the Ronkin function implies that
\(
p\in\Delta\cap\mathbb{Z}^n .
\)
Suppose that $p$ is not a vertex of $\Delta$. Then $p$ can be written as a
convex combination of vertices
\[
p=\sum_{i=1}^k\lambda_i\alpha_i,
\qquad
\alpha_i\in A,
\qquad
\lambda_i>0,
\qquad
\sum_{i=1}^k\lambda_i=1.
\]
Choose a vector $v\in\mathbb{R}^n$ such that
\(
\langle\alpha_1,v\rangle>
\langle\alpha_i,v\rangle\)
for \(i\ge2 .
\)
Such a vector exists because $\alpha_1$ is a vertex of the polytope.
Fix $x_0\in E$ and consider the ray
\(
x_s=x_0+s v,\) with 
\( s\ge 0,
\)
where $v\in\mathbb{R}^n$ is a fixed vector chosen so that
\[
\langle\alpha_1,v\rangle>
\langle\alpha,v\rangle
\qquad
\text{for all }\alpha\in\mathrm{Vert}(\Delta),\ \alpha\neq\alpha_1 .
\]
Since $N_{f_t}$ is affine on $E$, we obtain
\(
N_{f_t}(x_s)
=
\langle p,x_s\rangle+\beta
=
\langle p,x_0\rangle+\beta
+
s\langle p,v\rangle .
\)
Dividing by $s$ yields
\[
\frac{N_{f_t}(x_s)}{s}
=
\frac{\langle p,x_0\rangle+\beta}{s}
+
\langle p,v\rangle .
\]
Taking the limit as $s\to\infty$ gives
\[
\lim_{s\to\infty}\frac{N_{f_t}(x_s)}{s}
=
\langle p,v\rangle .
\]
On the other hand we consider the function
\[
F_t(x)=\max_{\alpha\in\mathrm{Vert}(\Delta)}
\{\langle\alpha,x\rangle+\nu(\alpha)\log t\}.
\]
Evaluating $F_t$ along the ray $x_s$ gives
\[
F_t(x_s)
=
\max_{\alpha\in\mathrm{Vert}(\Delta)}
\{\langle\alpha,x_0\rangle+\nu(\alpha)\log t+s\langle\alpha,v\rangle\}.
\]
Since $\langle\alpha_1,v\rangle$ is strictly larger than
$\langle\alpha,v\rangle$ for all other vertices $\alpha$, there exists
$s_0>0$ such that for all $s\ge s_0$ the maximum in the previous
expression is attained at $\alpha_1$ (see Lemma \ref{lem5}). Hence for $s\ge s_0$ we have
\(
F_t(x_s)
=
\langle\alpha_1,x_0\rangle+\nu(\alpha_1)\log t
+
s\langle\alpha_1,v\rangle .
\)
Dividing by $s$ gives
\[
\frac{F_t(x_s)}{s}
=
\frac{\langle\alpha_1,x_0\rangle}{s}
+
\frac{\nu(\alpha_1)\log t}{s}
+
\langle\alpha_1,v\rangle .
\]
Since
\[
\frac{\langle\alpha_1,x_0\rangle}{s}\to0,
\qquad
\frac{\log t}{s}\to0
\quad\text{as }s\to\infty,
\]
it follows that
\[
\lim_{s\to\infty}\frac{F_t(x_s)}{s}
=
\langle\alpha_1,v\rangle .
\]
We now use the fact that the difference between the Ronkin function and
$F_t$ remains uniformly bounded. More precisely, there exists a constant
$C>0$ independent of $x$ such that
\[
|N_{f_t}(x)-F_t(x)|\le C
\qquad\text{for all }x\in\mathbb{R}^n .
\]
Applying this estimate to the points $x_s$ yields
\(
|N_{f_t}(x_s)-F_t(x_s)|\le C .
\)
Dividing by $s$ gives
\[
\left|
\frac{N_{f_t}(x_s)}{s}
-
\frac{F_t(x_s)}{s}
\right|
\le
\frac{C}{s}.
\]
Letting $s\to\infty$ we obtain
\[
\lim_{s\to\infty}\frac{N_{f_t}(x_s)}{s}
=
\lim_{s\to\infty}\frac{F_t(x_s)}{s}.
\]
Combining the two limits computed above yields
\[
\langle p,v\rangle
=
\langle\alpha_1,v\rangle .
\]
However
\[
\langle p,v\rangle
=
\sum_{i=1}^k\lambda_i\langle\alpha_i,v\rangle
<
\langle\alpha_1,v\rangle ,
\]
which is a contradiction.
Therefore $p$ must belong to $\mathrm{Vert}(\Delta)$. Hence every
linearity domain of the Ronkin function has slope equal to a vertex
of the Newton polytope. In particular no new linearity domains with
slopes corresponding to interior lattice points of $\Delta$ can appear
for sufficiently small $t$.
\end{proof}


\subsection{Uniform Lower Bounds for Logarithmic Integrals of Trigonometric Polynomials}

In the comparison between the Ronkin function and the convex function
$F_t$, it is convenient to factor out the dominant exponential term from
the Laurent polynomial evaluated on the logarithmic torus. This leads to
a trigonometric polynomial whose coefficients depend on the parameters
$x$ and $t$. Although the coefficients vary with these parameters, they
remain confined to a compact set determined only by the original
coefficients of the Laurent polynomial. The following lemma provides a
uniform lower bound for the logarithmic average of the absolute value of
this trigonometric polynomial. This estimate plays a crucial role in
showing that the Ronkin function differs from $F_t$ by a bounded amount 
independently of $x$ and $t$.

\vspace{0.1cm} 
Let
\[
g_{x,t}(\theta)
=
\xi_{\alpha_0}e^{i\langle\alpha_0,\theta\rangle}
+
\sum_{\alpha\neq\alpha_0}
\xi_\alpha
e^{\ell_{\alpha,t}(x)-F_t(x)}
e^{i\langle\alpha,\theta\rangle},
\qquad 
\theta\in[0,2\pi]^n,
\]
where
\[
\ell_{\alpha,t}(x)=\langle\alpha,x\rangle+\nu(\alpha)\log t,
\qquad 
F_t(x)=\max_{\beta\in A}\ell_{\beta,t}(x),
\]
and $A=\mathrm{Vert}(\Delta)$.  

\begin{lemma}\label{lem7}
With the above notation, we prove that there exists a constant $C_1>0$ depending only on the finite
set $A$ and on the coefficients $\{\xi_\alpha\}_{\alpha\in A}$ such that
\[
\frac{1}{(2\pi)^n}
\int_{[0,2\pi]^n}
\log|g_{x,t}(\theta)|\,d\theta
\ge -C_1
\]
for all $x\in\mathbb{R}^n$ and all $t\in(0,1/e]$.
\end{lemma}

\begin{proof}
First observe that for every $\alpha\neq\alpha_0$ one has
\[
\ell_{\alpha,t}(x)-F_t(x)\le 0.
\]
Hence
\[
0\le e^{\ell_{\alpha,t}(x)-F_t(x)}\le 1 .
\]
Define
\[
c_\alpha(x,t)=\xi_\alpha e^{\ell_{\alpha,t}(x)-F_t(x)}
\qquad (\alpha\neq\alpha_0).
\]
Then
\[
|c_\alpha(x,t)|\le |\xi_\alpha|
\]
for all $x$ and $t$. Therefore the trigonometric polynomial
\[
g_{x,t}(\theta)
=
\xi_{\alpha_0}e^{i\langle\alpha_0,\theta\rangle}
+
\sum_{\alpha\neq\alpha_0}
c_\alpha(x,t)e^{i\langle\alpha,\theta\rangle}
\]
has coefficients lying in the compact set
\[
K=\left\{(c_\alpha)_{\alpha\in A\setminus\{\alpha_0\}}
\in\mathbb{C}^{|A|-1}\; ;\;
|c_\alpha|\le |\xi_\alpha|
\right\}.
\]

Introduce the family of trigonometric polynomials
\[
G_c(\theta)
=
\xi_{\alpha_0}e^{i\langle\alpha_0,\theta\rangle}
+
\sum_{\alpha\neq\alpha_0}
c_\alpha e^{i\langle\alpha,\theta\rangle},
\qquad c\in K.
\]
Then
\[
g_{x,t}(\theta)=G_{c(x,t)}(\theta)
\]
for some $c(x,t)\in K$.

Consider the function
\[
\Phi(c)
=
\frac{1}{(2\pi)^n}
\int_{[0,2\pi]^n}
\log|G_c(\theta)|\,d\theta .
\]
The function $\Phi$ is well defined for every $c\in K$.  
Indeed $G_c$ is not identically zero because the coefficient of the
monomial $e^{i\langle\alpha_0,\theta\rangle}$ is always $\xi_{\alpha_0}\neq0$.
Hence $G_c$ is a non-zero trigonometric polynomial and
$\log|G_c(\theta)|$ is integrable over $[0,2\pi]^n$.

We now show that $\Phi$ is continuous on $K$.  
For each $\theta$ the function $c\mapsto G_c(\theta)$ is linear, hence
continuous. Therefore $c\mapsto |G_c(\theta)|$ is continuous, and so
is $c\mapsto\log|G_c(\theta)|$ on the set where $G_c(\theta)\neq0$.
Since $G_c$ is not identically zero, the set of $\theta$ for which
$G_c(\theta)=0$ has measure zero. Using dominated convergence one
obtains that $\Phi$ is continuous on $K$.

Since $K$ is compact and $\Phi$ is continuous, the function $\Phi$
attains its minimum on $K$. Hence there exists a constant $C_1>0$
such that
\[
\Phi(c)\ge -C_1
\qquad\text{for all }c\in K.
\]

Since $c(x,t)\in K$ for all $x$ and $t$, it follows that
\[
\frac{1}{(2\pi)^n}
\int_{[0,2\pi]^n}
\log|g_{x,t}(\theta)|\,d\theta
=
\Phi(c(x,t))
\ge -C_1 .
\]

Thus the constant $C_1$ depends only on the finite set $A$ and on the
coefficients $\{\xi_\alpha\}_{\alpha\in A}$, and is therefore
independent of $x$ and $t$.
\end{proof}


\medskip


Let $\Delta\subset\mathbb{R}^n$ be a polytope and let 
$\mathrm{Vert}(\Delta)$ denote its set of vertices. Fix a point 
$x_0\in\mathbb{R}^n$ and a vector $v\in\mathbb{R}^n$, and consider the ray
\[
x_s=x_0+s v,
\qquad s\ge 0 .
\]
Let $\nu:\mathrm{Vert}(\Delta)\to\mathbb{R}$ be a function and define
\[
F_t(x)=
\max_{\alpha\in\mathrm{Vert}(\Delta)}
\{\langle\alpha,x\rangle+\nu(\alpha)\log t\}.
\]
Evaluating $F_t$ along the ray $x_s$ gives
\[
F_t(x_s)
=
\max_{\alpha\in\mathrm{Vert}(\Delta)}
\{\langle\alpha,x_0\rangle+\nu(\alpha)\log t+s\langle\alpha,v\rangle\}.
\]
For each $\alpha\in\mathrm{Vert}(\Delta)$ define the function
\(
g_\alpha(s)
=
\langle\alpha,x_0\rangle+\nu(\alpha)\log t+s\langle\alpha,v\rangle,\)\,
\( s>0 .
\)
Then
\[
F_t(x_s)=\max_{\alpha\in\mathrm{Vert}(\Delta)} g_\alpha(s).
\]
Assume that there exists a vertex $\alpha_1\in\mathrm{Vert}(\Delta)$
such that
\[
\langle\alpha_1,v\rangle>
\langle\alpha,v\rangle
\qquad
\text{for all }\alpha\in\mathrm{Vert}(\Delta),\ \alpha\neq\alpha_1 .
\]

\begin{lemma}\label{lem5}
With the above notation, 
 there exists $s_0>0$ such that for all $s\ge s_0$ we have 
\(
F_t(x_s)=g_{\alpha_1}(s).
\)
 \end{lemma}
 
 \begin{proof}
Let $\alpha\in\mathrm{Vert}(\Delta)$ with $\alpha\neq\alpha_1$ and
consider the difference
\(
g_{\alpha_1}(s)-g_\alpha(s).
\)
Substituting the definitions gives
\[
g_{\alpha_1}(s)-g_\alpha(s)
=
\langle\alpha_1-\alpha,x_0\rangle
+
(\nu(\alpha_1)-\nu(\alpha))\log t
+
s(\langle\alpha_1,v\rangle-\langle\alpha,v\rangle).
\]
Define
\(
\delta_\alpha=\langle\alpha_1,v\rangle-\langle\alpha,v\rangle .
\)
By assumption one has $\delta_\alpha>0$ for every $\alpha\neq\alpha_1$.
Hence
\[
g_{\alpha_1}(s)-g_\alpha(s)
=
A_\alpha
+
B_\alpha\log t
+
s\,\delta_\alpha ,
\]
where
\(
A_\alpha=\langle\alpha_1-\alpha,x_0\rangle,\)
and \(
B_\alpha=\nu(\alpha_1)-\nu(\alpha).
\)
Since $\delta_\alpha>0$, the term $s\,\delta_\alpha$ grows linearly
to $+\infty$ as $s\to\infty$, while the remaining terms grow at most
logarithmically. Consequently
\[
g_{\alpha_1}(s)-g_\alpha(s)\to+\infty
\qquad\text{as }s\to\infty .
\]
Therefore there exists a number $s_\alpha>0$ such that
\(
g_{\alpha_1}(s)>g_\alpha(s)\)
for all \(s\ge s_\alpha .
\)
Since the set $\mathrm{Vert}(\Delta)$ is finite, the set
\[
\{s_\alpha\mid \alpha\in\mathrm{Vert}(\Delta),\ \alpha\neq\alpha_1\}
\]
is finite. Define
\(
s_0=\max_{\alpha\neq\alpha_1} s_\alpha .
\)
Then for every $\alpha\neq\alpha_1$ and every $s\ge s_0$ one has
\(
g_{\alpha_1}(s)>g_\alpha(s).
\)
Consequently for all $s\ge s_0$
\[
\max_{\alpha\in\mathrm{Vert}(\Delta)} g_\alpha(s)
=
g_{\alpha_1}(s).
\]
Substituting the definition of $g_{\alpha_1}$ yields
\[
F_t(x_s)
=
\langle\alpha_1,x_0\rangle
+
\nu(\alpha_1)\log t
+
s\langle\alpha_1,v\rangle
\qquad\text{for all }s\ge s_0 .
\]

Thus there exists $s_0>0$ such that for every $s\ge s_0$
the maximum in the definition of $F_t(x_s)$ is attained at
$\alpha_1$, which proves the lemma.
\end{proof}


\medskip
 
\section{Stability of the Spine and Tropical Subdivision}

In tropical geometry, the spine of an amoeba provides a piecewise-linear
skeleton capturing the essential combinatorial features of the amoeba.
When a Laurent polynomial undergoes a degeneration governed by a weight
function on its exponents, the associated amoebas deform and their spines
are expected to approach the tropical hypersurface determined by the
corresponding tropical polynomial. A natural question is whether this
convergence also preserves the combinatorial structure of the spine.
In particular, one may ask whether the subdivision of the Newton polytope
dual to the spine stabilizes and eventually coincides with the regular
subdivision induced by the tropical limit. The following theorem shows
that this is indeed the case for maximally sparse Laurent polynomials:
for sufficiently small values of the degeneration parameter, the spine
of the amoeba has the same combinatorial type as the tropical hypersurface,
and the corresponding dual subdivision of the Newton polytope becomes
identical to the one determined by the tropical polynomial.

\begin{theorem}[Tropical stability for maximally sparse polynomials]
Let
\[
f(z)=\sum_{\alpha\in \mathrm{Vert}(\Delta_f)} a_\alpha z^\alpha
\]
be a maximally sparse Laurent polynomial, that is,
\(
\supp(f)=\mathrm{Vert}(\Delta_f).
\)
Let $\nu:\mathrm{Vert}(\Delta_f)\to\mathbb{R}$ be a weight function and
consider the family
\[
f_t(z)=\sum_{\alpha\in \mathrm{Vert}(\Delta_f)} a_\alpha t^{\nu(\alpha)} z^\alpha,
\qquad t\in(0,1].
\]
Let $\Gamma_t$ be the spine of the amoeba $\mathscr{A}_{f_t}$ and let
$\Gamma_\infty$ be the tropical hypersurface defined by the tropical
polynomial
\[
F(x)=\max_{\alpha\in \mathrm{Vert}(\Delta_f)}
\{\langle\alpha,x\rangle-\nu(\alpha)\}.
\]
Denote by $\tau_\infty$ the regular subdivision of the Newton polytope
$\Delta_f$ dual to $\Gamma_\infty$.
Then there exists a number $t_0>0$ such that for every $t\in(0,t_0)$
the spine $\Gamma_t$ has the same combinatorial type as the tropical
hypersurface $\Gamma_\infty$. Equivalently, the subdivision of the
Newton polytope dual to $\Gamma_t$ coincides with $\tau_\infty$ for all
sufficiently small $t$.
\end{theorem}

\begin{proof}
Let $N_{f_t}$ denote the Ronkin function associated with the polynomial
$f_t$. It is known that $N_{f_t}$ is a convex function on $\mathbb{R}^n$
which is affine linear on each connected component of the complement
$\mathbb{R}^n\setminus\mathscr{A}_{f_t}$. The spine $\Gamma_t$ of the
amoeba is the set of points where the piecewise affine linear function approximating
Ronkin function is not
differentiable.
A fundamental property of Ronkin functions states that
\[
N_{f_t}(x)+(\log t)\,C \longrightarrow
F(x)
\]
locally uniformly on $\mathbb{R}^n$ as $t\to 0$, where $C$ is a
constant depending only on the coefficients of the polynomial and
\[
F(x)=\max_{\alpha\in\mathrm{Vert}(\Delta_f)}
\{\langle\alpha,x\rangle-\nu(\alpha)\}
\]
is the tropical polynomial associated with the degeneration.
Consequently the spines $\Gamma_t$ converge, in the Hausdorff topology
on compact subsets of $\mathbb{R}^n$, to the tropical hypersurface
$\Gamma_\infty$ defined as the corner locus of $F$.

The Ronkin function is approximated by a piecewise affine linear function  and its affine pieces have
gradients equal to the orders of complement components of the amoeba.
These orders belong to the finite set
\[
\Delta_f\cap\mathbb{Z}^n.
\]
Therefore the normal vectors of the faces of the spine $\Gamma_t$
belong to the finite set
\[
\{\beta_i-\beta_j \mid
\beta_i,\beta_j\in\Delta_f\cap\mathbb{Z}^n\}.
\]
Hence the slopes of all faces of $\Gamma_t$ belong to a finite set of
rational directions.

Since $\Gamma_t$ converges to the tropical hypersurface $\Gamma_\infty$
and the possible slopes of the faces belong to a finite set, the
combinatorial type of the polyhedral complex $\Gamma_t$ can change
only finitely many times as $t$ varies. If infinitely many changes
occurred as $t\to0$, new faces with new supporting hyperplanes would
have to appear infinitely often, which is impossible because the set
of possible directions is finite.
The role of the maximally sparse assumption now becomes essential.
Because the support of the polynomial consists only of the vertices
of the Newton polytope, the tropical polynomial involves only the
affine functions
\[
\langle\alpha,x\rangle-\nu(\alpha),
\qquad \alpha\in\mathrm{Vert}(\Delta_f).
\]
Consequently the regular subdivision $\tau_\infty$ of the Newton
polytope induced by the lifting
\[
\alpha\longmapsto (\alpha,\nu(\alpha))
\]
uses only these vertices. In particular, no lattice point in the
interior of $\Delta_f$ can appear as a vertex of the subdivision.
Therefore the combinatorial structure of the tropical hypersurface
$\Gamma_\infty$ is completely determined by the vertices of the
Newton polytope.

Since $\Gamma_t$ converges to $\Gamma_\infty$ and only finitely many
polyhedral types are possible, and by Theorem \ref{prop1} 
 there exists a number $t_0>0$ such that
for every $t\in(0,t_0)$ the spine $\Gamma_t$ has the same combinatorial
type as $\Gamma_\infty$.

\end{proof}


\section{Stability of Solid Amoebas in Maximally Sparse Degenerations}

In the study of amoebas associated with families of Laurent polynomials,
an important question concerns the behavior of the topology of the amoeba
under degenerations of the coefficients. In particular, when the support
of the polynomial consists only of the vertices of its Newton polytope,
the polynomial is said to be maximally sparse, and the geometry of the
associated amoebas exhibits strong rigidity properties. When such a
polynomial is deformed by introducing weights on its monomials, one
obtains a family of hypersurfaces whose images under the logarithmic map
produce a family of amoebas depending on a parameter $t$. A natural
problem is to understand whether the qualitative structure of these
amoebas remains stable as the parameter varies. The following theorem
shows that the property of being solid, meaning that the complement of
the amoeba has exactly as many connected components as the number of
vertices of the Newton polytope, is preserved throughout the entire
degeneration. More precisely, the set of parameters for which the amoeba
is solid is both open and closed in the interval $(0,1/e]$, and therefore
coincides with the whole interval.

\vspace{0.1cm}
 
\begin{theorem}
Let $f$ be a maximally sparse Laurent polynomial, that is,
\(
\supp(f)=\mathrm{Vert}(\Delta_f).
\)
Consider the family
\[
f_t(z)=\sum_{\alpha\in\supp(f)} \xi_\alpha t^{\nu(\alpha)}z^\alpha,
\qquad t\in(0,1/e],
\]
and denote by $V_{f_t}$ the hypersurface in $(\mathbb{C}^*)^n$
defined by $f_t$. Let
\(
V_{\infty,t}=H_t(V_{f_t})
\)
be the corresponding family of $J_t$-holomorphic hypersurfaces and
let
\(
\mathscr{A}_t=\mathscr{A}_{H_t(V_{f_t})}
\)
be their amoebas.

Define
\[
\mathscr{S}=
\left\{
t\in(0,1/e]\mid \mathscr{A}_{H_t(V_{f_t})}\ \text{is solid}
\right\}.
\]
Then $\mathscr{S}$ is nonempty, open, and closed in $(0,1/e]$.
Consequently
\(
\mathscr{S}=(0,1/e].
\)
\end{theorem}

\begin{proof}
For every $t\in(0,1/e]$ the Laurent polynomial
\[
f_t(z)=\sum_{\alpha\in \mathrm{Vert}(\Delta_f)} \xi_\alpha t^{\nu(\alpha)}z^\alpha
\]
has the same support as $f$, namely the set of vertices of the Newton
polytope $\Delta_f$. Hence the Newton polytope of $f_t$ coincides with
$\Delta_f$ for all $t$.

Let $V_{f_t}\subset (\mathbb C^*)^n$ be the hypersurface defined by $f_t$
and let
\[
\Log:(\mathbb C^*)^n\to\mathbb R^n,\qquad 
\Log(z_1,\dots,z_n)=(\log|z_1|,\dots,\log|z_n|)
\]
be the logarithmic map. The amoeba of a hypersurface $V$ is defined by
\[
\mathscr A_V=\Log(V).
\]
The rescaling map $H_t$ appearing in the tropical degeneration induces a
homeomorphism of $(\mathbb C^*)^n$ which corresponds to a linear
transformation of $\mathbb R^n$ under the logarithmic map. Therefore
there exists a linear homeomorphism $L_t:\mathbb R^n\to\mathbb R^n$ such
that
\[
\mathscr A_{H_t(V_{f_t})}=L_t(\mathscr A_{V_{f_t}}).
\]
Since linear homeomorphisms preserve connectedness properties of the
complement, the amoeba $\mathscr A_{H_t(V_{f_t})}$ is solid if and only
if $\mathscr A_{V_{f_t}}$ is solid. Consequently the property defining
$\mathscr S$ is equivalent to the solidness of $\mathscr A_{V_{f_t}}$.

The complement components of the amoeba are in bijection with the
linearity domains of the Ronkin function
\[
N_{f_t}(x)=\frac{1}{(2\pi)^n}\int_{[0,2\pi]^n}
\log|f_t(e^{x+i\theta})|\,d\theta .
\]
The Ronkin function is convex on $\mathbb R^n$ and affine precisely on
the connected components of $\mathbb R^n\setminus\mathscr A_{V_{f_t}}$.
Moreover its gradient on each linearity domain is an integer vector
belonging to the set
\[
\Delta_f\cap\mathbb Z^n .
\]

Since the polynomial $f_t$ contains monomials only at the vertices of
$\Delta_f$, Proposition~\ref{prop1} implies that for sufficiently small
$t>0$ the Ronkin function $N_{f_t}$ has linearity domains only with
slopes corresponding to the vertices of $\Delta_f$. Hence the number of
connected components of $\mathbb R^n\setminus\mathscr A_{V_{f_t}}$ is
equal to the number of vertices of $\Delta_f$. Therefore the amoeba is
solid for all sufficiently small values of $t$. This shows that
$\mathscr S$ is nonempty.

We now prove that $\mathscr S$ is open in $(0,1/e]$. The coefficients of
$f_t$ depend continuously on $t$, hence the Ronkin function $N_{f_t}$
depends continuously on $t$ uniformly on compact subsets of
$\mathbb R^n$. The number of linearity domains of a convex function is
locally constant under small perturbations provided that no new slopes
appear. Since the support of the polynomial remains fixed and equal to
$\mathrm{Vert}(\Delta_f)$, no new slopes can arise under sufficiently
small perturbations of $t$. Consequently if the amoeba is solid for some
$t_0$, it remains solid for all $t$ sufficiently close to $t_0$, which
shows that $\mathscr S$ is open.

To prove that $\mathscr S$ is closed, let $\{t_k\}\subset\mathscr S$ be a
sequence converging to some $t_0\in(0,1/e]$. The polynomials $f_{t_k}$
converge coefficientwise to $f_{t_0}$, and the corresponding Ronkin
functions $N_{f_{t_k}}$ converge uniformly on compact subsets to
$N_{f_{t_0}}$. The linearity domains of the Ronkin function can only
disappear in the limit but cannot create new slopes outside
$\mathrm{Vert}(\Delta_f)$. Since each $f_{t_k}$ has exactly
$\#\mathrm{Vert}(\Delta_f)$ complement components, the same holds for
$f_{t_0}$. Thus the amoeba $\mathscr A_{V_{f_{t_0}}}$ remains solid and
$t_0\in\mathscr S$. Hence $\mathscr S$ is closed in $(0,1/e]$.

Since $(0,1/e]$ is connected and $\mathscr S$ is nonempty, open, and
closed in $(0,1/e]$, it follows that
\[
\mathscr S=(0,1/e].
\]
\end{proof}


\section{Compact Complement Components and the Newton Subdivision}

The topology of the complement of an amoeba is closely related to the
convex–geometric structure encoded by the Newton polytope of the defining
Laurent polynomial. In particular, the Ronkin function and the spine of the
amoeba provide a bridge between the analytic geometry of the hypersurface
and the combinatorics of a convex subdivision of the Newton polytope.
The spine determines a regular subdivision of $\Delta_f$, and the affine
linearity domains of the Ronkin function correspond to lattice points of
the polytope via the order map of amoebas. A natural question is to
identify precisely which lattice points give rise to compact connected
components of the complement of the amoeba. The following lemma answers
this question by establishing a direct correspondence between compact
complement components of $\mathbb{R}^n\setminus\mathscr{A}_f$ and the
interior vertices of the subdivision of the Newton polytope dual to the
spine of the amoeba.

\medskip

\begin{lemma}\label{lem4}
Let
\[
f(z)=\sum_{\alpha\in\supp(f)} a_\alpha z^\alpha
\]
be a Laurent polynomial with Newton polytope $\Delta_f\subset\mathbb{R}^n$.
Let $\mathscr{A}_f$ be the amoeba of the hypersurface $V_f\subset(\mathbb{C}^*)^n$,
let $N_f$ be the Ronkin function of $f$, and let $\Gamma$ be the spine of the amoeba.
Denote by $\tau$ the convex subdivision of $\Delta_f$ dual to the spine $\Gamma$.

Then a lattice point $\alpha\in\Delta_f\cap\mathbb{Z}^n$ corresponds to a compact
connected component of the complement $\mathbb{R}^n\setminus\mathscr{A}_f$
if and only if $\alpha$ is an interior vertex of the subdivision $\tau$ of the Newton polytope.
\end{lemma}

\begin{proof}
The results  of 
Forsberg, Passare and Tsikh in \cite{FPT-00}  asserts that the gradient
of $N_f$ on a complement component $E_\alpha$ is equal to a lattice point
\(
\nabla N_f(x)=\alpha,\) with  \( x\in E_\alpha,
\)
where $\alpha\in\Delta_f\cap\mathbb{Z}^n$.
This defines the order map
\[
\mathrm{ord}:\{\text{components of }\mathbb{R}^n\setminus\mathscr{A}_f\}
\longrightarrow
\Delta_f\cap\mathbb{Z}^n .
\]
The Ronkin function can therefore be written as the maximum of finitely
many affine functions
\[
N_f(x)=\max_{\alpha\in A'}
\{\langle\alpha,x\rangle+c_\alpha\},
\]
where $A'\subset\Delta_f\cap\mathbb{Z}^n$ is the image of the order map.
The spine $\Gamma$ of the amoeba is the corner locus of this convex
piecewise affine function.
The polyhedral complex $\Gamma$ is dual to a convex subdivision
$\tau$ of the Newton polytope $\Delta_f$. The duality is defined as
follows. Each connected component $E_\alpha$ of the complement of the
amoeba corresponds to a vertex of the subdivision $\tau$, namely the
lattice point $\alpha$. The affine region of $N_f$ with gradient
$\alpha$ corresponds to this vertex.

Assume first that the complement component $E_\alpha$ is compact.
Since $N_f$ is affine on $E_\alpha$, the function
\[
x\mapsto \langle\alpha,x\rangle+c_\alpha
\]
strictly dominates all other affine pieces of the Ronkin function
inside $E_\alpha$. Because $E_\alpha$ is bounded, the affine function
associated with $\alpha$ must be surrounded by other affine pieces in
all directions. In the dual picture this means that the vertex
$\alpha$ of the subdivision is surrounded by other cells of the
subdivision. Therefore $\alpha$ lies in the interior of the Newton
polytope and is an interior vertex of the subdivision $\tau$.

Conversely, suppose that $\alpha$ is an interior vertex of the
subdivision $\tau$. By duality with the spine $\Gamma$, the vertex
$\alpha$ corresponds to a bounded region of linearity of the Ronkin
function. Indeed, if $\alpha$ lies in the interior of $\Delta_f$, the
cells of the subdivision adjacent to $\alpha$ form a complete
polyhedral fan around that point. The dual cells of the spine form a
closed polyhedral region in $\mathbb{R}^n$, which corresponds to a
bounded domain where the Ronkin function has gradient $\alpha$.

This bounded domain is precisely the complement component
$E_\alpha$ of the amoeba. Hence $E_\alpha$ is compact.
Thus the compact connected components of the complement of the amoeba
are in one-to-one correspondence with the interior vertices of the
Newton subdivision dual to the spine.
\end{proof}

\medskip

This lemma provides a precise geometric bridge between the topology of
amoebas and the combinatorics of Newton polytopes. Compact complement
components correspond exactly to interior vertices of the dual
subdivision. Consequently, if the Newton subdivision has no interior
vertices, the amoeba cannot develop compact complement components.
This principle plays a central role in the study of solid amoebas and
in the proof of the main theorem of the paper.

 
\section{Classification of Amoeba Complement Topology}

The topology of the complement of an amoeba reflects deep interactions
between the analytic properties of a Laurent polynomial and the
combinatorial structure of its Newton polytope. In particular, the set
of lattice points appearing in the support of the polynomial determines
which linearity domains of the Ronkin function may occur and therefore
which connected components can appear in the complement of the amoeba.
Through the order map of amoebas and the convex geometry of the Newton
polytope, one obtains strong constraints on the possible topology of
$\mathbb{R}^n\setminus\mathscr A_f$. The following theorem summarizes
this relationship by describing three distinct regimes according to the
position of the support of the polynomial relative to its Newton
polytope. Depending on whether the support consists only of the
vertices, lies entirely on the boundary, or contains interior lattice
points, the complement of the amoeba exhibits qualitatively different
topological behaviors.

\medskip


\begin{theorem}[Three regimes for amoeba topology]
Let
\[
f(z)=\sum_{\alpha\in\supp(f)} a_\alpha z^\alpha
\]
be a Laurent polynomial with Newton polytope $\Delta_f$ and amoeba
$\mathscr A_f$. Then the topology of the complement
$\mathbb R^n\setminus\mathscr A_f$ falls into one of the following
three regimes.
\begin{itemize}
\item[(i)] If $\supp(f)=\mathrm{Vert}(\Delta_f)$, that is, the polynomial is
maximally sparse, then the amoeba is solid and the connected
components of $\mathbb R^n\setminus\mathscr A_f$ are in bijection with
the vertices of $\Delta_f$.

\item[(ii)] If the support of the polynomial is contained in the boundary of the
Newton polytope,
\[
\supp(f)\subset\partial\Delta_f\cap\mathbb Z^n ,
\]
then the amoeba has no compact complement components. In this case
every connected component of $\mathbb R^n\setminus\mathscr A_f$ is
unbounded and its order belongs to
$\partial\Delta_f\cap\mathbb Z^n$.

\item[(iii)]If the support contains at least one interior lattice point of the
Newton polytope, that is,
\[
\supp(f)\cap\mathrm{Int}(\Delta_f)\neq\varnothing,
\]
then the amoeba may develop compact complement components whose
orders correspond to interior lattice points of $\Delta_f$.
\end{itemize}
\end{theorem}

\begin{proof}
Let
\[
f(z)=\sum_{\alpha\in A} a_\alpha z^\alpha ,\qquad A=\supp(f)\subset\mathbb Z^n,
\]
and denote by $\Delta_f=\mathrm{Conv}(A)$ its Newton polytope. The amoeba
of the hypersurface
\[
V_f=\{z\in(\mathbb C^*)^n\mid f(z)=0\}
\]
is defined as
\[
\mathscr A_f=\Log(V_f),\qquad 
\Log(z_1,\dots,z_n)=(\log|z_1|,\dots,\log|z_n|).
\]
A fundamental result of Forsberg--Passare--Tsikh states that every
connected component $E$ of the complement
\(
\mathbb R^n\setminus\mathscr A_f
\)
admits an associated integer vector
\(
\mathrm{ord}(E)\in\Delta_f\cap\mathbb Z^n
\)
called its order, and the map
\[
\mathrm{ord}:\pi_0(\mathbb R^n\setminus\mathscr A_f)
\longrightarrow
\Delta_f\cap\mathbb Z^n
\]
is injective. Consequently the number of connected components of the
complement of the amoeba is bounded above by the number of lattice
points of $\Delta_f$. Moreover, if $\alpha$ is a vertex of $\Delta_f$,
then there exists a connected component $E_\alpha$ such that
$\mathrm{ord}(E_\alpha)=\alpha$.

We now analyze the three possible situations for the support of the
polynomial.
Assume first that
\(
A=\mathrm{Vert}(\Delta_f).
\)
In this case the support of the polynomial consists exactly of the
vertices of its Newton polytope. For $x\in\mathbb R^n$ consider
\[
|f(e^x)|\le\sum_{\alpha\in A}|a_\alpha|e^{\langle\alpha,x\rangle}.
\]
Define the convex piecewise linear function
\[
F(x)=\max_{\alpha\in A}
\{\langle\alpha,x\rangle+\log|a_\alpha|\}.
\]
The Ronkin function
\[
N_f(x)=\frac{1}{(2\pi)^n}\int_{[0,2\pi]^n}\log|f(e^{x+i\theta})|d\theta
\]
satisfies
\[
|N_f(x)-F(x)|\le C
\]
for some constant $C$. The function $F$ has linearity domains whose
gradients are precisely the vertices of $\Delta_f$. Since the Ronkin
function differs from $F$ by a bounded quantity, its gradients cannot
take values outside the convex hull of these slopes, and the only
possible integer slopes are the vertices themselves (see Theorem \ref{prop1}). Therefore the
linearity domains of $N_f$ correspond exactly to the vertices of
$\Delta_f$, and hence the complement components of the amoeba are in
bijection with these vertices. In particular the number of complement
components equals the number of vertices of $\Delta_f$, which means
that the amoeba is solid.

Assume now that
\(
A\subset\partial\Delta_f\cap\mathbb Z^n .
\)
Thus the polynomial contains no monomials corresponding to interior
lattice points of $\Delta_f$. Suppose that a compact connected
component $E$ of $\mathbb R^n\setminus\mathscr A_f$ existed. Then its
order $\mathrm{ord}(E)$ would belong to $\Delta_f\cap\mathbb Z^n$ by the
order map theorem. Compact components correspond to interior slopes of
the Ronkin function, hence $\mathrm{ord}(E)$ would have to lie in the
interior of $\Delta_f$ (see Lemma \ref{lem4}). However interior lattice points do not belong to
the support of the polynomial. Since the Ronkin function is asymptotic
to the support function determined by the monomials of $f$, such slopes
cannot occur, which contradicts the existence of $E$. Therefore no
compact complement components exist and all components are unbounded.
Their orders necessarily lie in $\partial\Delta_f\cap\mathbb Z^n$.

Finally assume that
\(
A\cap\mathrm{Int}(\Delta_f)\neq\varnothing.
\)
Let $\beta$ be an interior lattice point belonging to the support.
Consider the term
\(
a_\beta z^\beta .
\)
When $x$ varies in $\mathbb R^n$, the quantity
\[
\langle\beta,x\rangle+\log|a_\beta|
\]
can dominate the other affine functions associated with the vertices
in a bounded region of $\mathbb R^n$. In such a region the Ronkin
function may develop a linearity domain whose slope is $\beta$.
Since $\beta$ lies in the interior of $\Delta_f$ (Lemma \ref{lem4}), the corresponding
component of $\mathbb R^n\setminus\mathscr A_f$ is bounded. Hence the
amoeba may develop compact complement components whose orders are
interior lattice points of $\Delta_f$.
These three possibilities exhaust all configurations of the support
of the polynomial relative to its Newton polytope, and therefore the
topology of the complement of the amoeba necessarily falls into one
of the three regimes described above.
\end{proof}


\section{Geometric Criteria Preventing Lattice Points from Appearing in the Dual Subdivision}


The geometry of amoebas is intimately related to the convex–analytic
structure of the Ronkin function and to the combinatorics of the Newton
polytope of the defining Laurent polynomial. In particular, the spine of
the amoeba encodes a piecewise–linear skeleton which reflects the dominant
affine pieces of the Ronkin function. This structure admits a natural
polyhedral interpretation through the theory of convex liftings. By
lifting lattice points of the Newton polytope to one dimension higher
according to the constants appearing in the affine pieces of the Ronkin
function, one obtains a lifted configuration whose lower convex hull
induces a regular polyhedral subdivision of the Newton polytope. As we know, this
subdivision is dual to the spine of the amoeba and provides a powerful
geometric framework for understanding the topology of the complement of
the amoeba. In particular, the vertices of the induced subdivision
correspond to those lattice points whose associated affine functions
contribute to the maximal pieces of the Ronkin function. The goal of this
section is to establish a geometric criterion that prevents certain
lattice points from appearing as vertices of this dual subdivision. This
criterion reveals how convex relations among lattice points of the support
restrict the possible affine pieces of the Ronkin function and therefore
control the topology of the amoeba complement. As an immediate consequence,
we obtain strong restrictions when the support of the polynomial lies
entirely on the boundary of its Newton polytope, showing that interior
lattice points cannot arise as orders of complement components of the
amoeba.

   Let 
\[
f(z)=\sum_{\beta\in\supp(f)} a_\beta z^\beta
\]
be a Laurent polynomial with Newton polytope $\Delta_f\subset\mathbb{R}^n$,
and let
\(
\mathscr{A}_f=\Log(V_f)
\)
be its amoeba. Denote by $N_f$ the Ronkin function of $f$ and by
\(
\Gamma
\)
the spine of the amoeba. The spine $\Gamma$ is the corner locus of the
convex piecewise affine function
\[
F(x)=\max_{\alpha\in A'}\{\langle\alpha,x\rangle+c_\alpha\},
\]
where $A'$ is the image of the order map and the constants $c_\alpha$
come from the affine pieces of the Ronkin function.
The lifted points
\(
(\alpha,-c_\alpha)\in\mathbb{R}^{n+1}
\)
define a convex lifting of the Newton polytope, and the lower convex
hull of these points induces a polyhedral subdivision
\(
\tau
\)
of $\Delta_f$ that is dual to the spine $\Gamma$.
The appearance of a lattice point $\alpha\in\Delta_f\cap\mathbb{Z}^n$
as a vertex of the subdivision $\tau$ is equivalent to the fact that
the lifted point $(\alpha,-c_\alpha)$ lies on the lower convex hull
of the lifted configuration. Geometrically this means that the affine
function
\(
\ell_\alpha(x)=\langle\alpha,x\rangle+c_\alpha
\)
appears as one of the maximal affine pieces of the convex function
$F(x)$.
The following theorem gives a geometric criterion preventing a lattice
point from appearing as a vertex of the dual subdivision.

\begin{theorem}
Let $\alpha\in\Delta_f\cap\mathbb{Z}^n$. Suppose there exist lattice
points $\beta_1,\ldots,\beta_k\in\supp(f)$ and nonnegative real
numbers $\lambda_1,\ldots,\lambda_k$ satisfying
\[
\sum_{i=1}^k\lambda_i=1,
\qquad
\alpha=\sum_{i=1}^k\lambda_i\beta_i,
\]
and such that the constants $c_\alpha$ satisfy
\[
c_\alpha \ge
\sum_{i=1}^k \lambda_i c_{\beta_i}.
\]
Then the lifted point $(\alpha,-c_\alpha)$ does not lie on the lower
convex hull of the lifted configuration. Consequently $\alpha$ is not
a vertex of the subdivision $\tau$ dual to the spine of the amoeba.
\end{theorem}

\begin{proof}
Consider the lifted points
\(
(\beta_i,-c_{\beta_i})\in\mathbb{R}^{n+1}.
\)
The convex combination determined by the coefficients $\lambda_i$
gives the point
\[
\left(
\sum_{i=1}^k\lambda_i\beta_i,
-
\sum_{i=1}^k\lambda_i c_{\beta_i}
\right)
=
\left(
\alpha,
-
\sum_{i=1}^k\lambda_i c_{\beta_i}
\right).
\]
By assumption one has
\[
c_\alpha\ge\sum_{i=1}^k\lambda_i c_{\beta_i},
\]
which implies
\[
-c_\alpha \le
-
\sum_{i=1}^k\lambda_i c_{\beta_i}.
\]

Therefore the lifted point $(\alpha,-c_\alpha)$ lies above or on the
hyperplane determined by the convex combination of the lifted points
$(\beta_i,-c_{\beta_i})$. In particular it cannot lie strictly below
this hyperplane.
Since the lower convex hull consists of points that lie below all
affine combinations of the lifted configuration, the point
$(\alpha,-c_\alpha)$ cannot belong to the lower convex hull. Hence it
does not correspond to a vertex of the induced subdivision $\tau$ of
$\Delta_f$.
\end{proof}

This criterion has a direct geometric interpretation. If a lattice
point $\alpha$ is a convex combination of other lattice points in the
support of the polynomial and if its lifting lies above the
corresponding convex combination of their lifted points, then the
affine function associated with $\alpha$ never dominates the Ronkin
function. Consequently $\alpha$ cannot appear as the order of a
complement component of the amoeba and does not occur as a vertex of
the dual subdivision.
A particularly important consequence appears when the support of the
polynomial is contained entirely in the boundary of the Newton
polytope.

\begin{corollary}
If $\supp(f)$ is contained in the boundary of the Newton polytope
$\Delta_f$, then every vertex of the subdivision $\tau$ induced by
the Ronkin function also lies on the boundary of $\Delta_f$.
Consequently no interior lattice point of $\Delta_f$ can appear as
the order of a complement component of the amoeba.
\end{corollary}

\begin{proof}
If $\alpha$ is an interior lattice point of $\Delta_f$, then it can be
written as a convex combination of lattice points lying on the
boundary of the polytope. Since all monomials of the polynomial lie
on the boundary, the lifted configuration consists only of boundary
points.
Applying the theorem above shows that the lifted point corresponding
to $\alpha$ cannot lie on the lower convex hull. Hence $\alpha$
cannot appear as a vertex of the subdivision $\tau$.
\end{proof}

This result shows that the geometry of the Newton polytope places
strong restrictions on the possible orders of complement components
of the amoeba. In particular, interior lattice points cannot produce
complement components when the polynomial is supported only on the
boundary of its Newton polytope.

\end{document}